\documentclass[a4paper,12pt]{amsart}
\usepackage{amssymb}
\usepackage{ifthen}
\usepackage{cite}
 \usepackage[dvips]{graphicx}
 
\nonstopmode \numberwithin{equation}{section}
\usepackage{amssymb}
\usepackage{ifthen}
\usepackage{graphicx}
\usepackage{amsmath}
\usepackage[T1]{fontenc} 
\usepackage[utf8]{inputenc}
\usepackage[usenames,dvipsnames]{color}
\usepackage{color}
\usepackage[english]{babel}
\usepackage{fancyhdr}
\usepackage{fancybox}
\usepackage{tikz}

\nonstopmode \numberwithin{equation}{section}
\theoremstyle{plain}

\newtheorem{conj}{Conjecture}

\theoremstyle{definition}
\newtheorem{defn}{Definition}[section]

\newtheorem{thm}{Theorem}[section]
\newtheorem{prob}{Problem}[section]
\newtheorem{cor}{Corollary}[section]

\newtheorem{prop}{Proposition}[section]
\newtheorem{rem}{Remark}[section]
\newtheorem{lem}{Lemma}[section]

\newcounter{minutes}
\divide\time by 60
\newcounter{hours}
\multiply\time by 60
\addtocounter{minutes}{-\time}

\newcounter {own}
\def\theown {\thesection       .\arabic{own}}

\newenvironment{pf}[1][]{%
 \vskip 3mm
 \noindent
 \ifthenelse{\equal{#1}{}}%
  {{\slshape Proof. }}%
  {{\slshape #1.} }%
 }%
{\qed\bigskip}

\newcounter{alphabet}

\theoremstyle{plain}
\newtheorem*{thmA}{Theorem A}
\newtheorem*{thmB}{Theorem B}
\newtheorem*{thmC}{Theorem C}
\newtheorem*{thmD}{Theorem D}
\newtheorem*{thmE}{Theorem E}
\newtheorem*{thmF}{Theorem F}
\newtheorem*{thmG}{Theorem G}

\newtheorem*{lemA}{Lemma A}

\def\be{\begin{equation}}
\def\ee{\end{equation}}

\newcommand{\bee}{\begin{enumerate}}
\newcommand{\eee}{\end{enumerate}}

\newcommand{\blem}{\begin{lem}}
\newcommand{\elem}{\end{lem}}
\newcommand{\bthm}{\begin{thm}}
\newcommand{\ethm}{\end{thm}}
\newcommand{\bcor}{\begin{cor}}
\newcommand{\ecor}{\end{cor}}
\newcommand{\beg}{\begin{examp}}
\newcommand{\eeg}{\end{examp}}
\newcommand{\begs}{\begin{examples}}
\newcommand{\eegs}{\end{examples}}

\newcommand{\bdefn}{\begin{defn}}
\newcommand{\edefn}{\end{defn}}

\newcommand{\bprob}{\begin{prob}}
\newcommand{\eprob}{\end{prob}}
\newcommand{\bei}{\begin{itemize}}
\newcommand{\eei}{\end{itemize}}

\newcommand{\bcon}{\begin{conj}}
\newcommand{\econ}{\end{conj}}
\newcommand{\bcons}{\begin{conjs}}
\newcommand{\econs}{\end{conjs}}
\newcommand{\bprop}{\begin{prop}}
\newcommand{\eprop}{\end{prop}}
\newcommand{\br}{\begin{rem}}
\newcommand{\er}{\end{rem}}
\newcommand{\brs}{\begin{rems}}
\newcommand{\ers}{\end{rems}}
\newcommand{\bo}{\begin{obser}}
\newcommand{\eo}{\end{obser}}
\newcommand{\bos}{\begin{obsers}}
\newcommand{\eos}{\end{obsers}}
\newcommand{\bpf}{\begin{pf}}
\newcommand{\epf}{\end{pf}}
\newcommand{\ba}{\begin{array}}
\newcommand{\ea}{\end{array}}
\newcommand{\beq}{\begin{eqnarray}}
\newcommand{\beqq}{\begin{eqnarray*}}
\newcommand{\eeq}{\end{eqnarray}}
\newcommand{\eeqq}{\end{eqnarray*}}

\begin{document}

\title{Bohr-Rogosinski Inequalities for bounded analytic functions}

\author{Molla Basir Ahamed}
\address{Molla Basir Ahamed, Department of Mathematics, Jadavpur University, Kolkata-700032, West Bengal,India.}
\email{mbahamed.math@jadavpuruniversity.in}
\author{Partha Pratim Roy}
\address{Partha Pratim Roy, Department of Mathematics, Jadavpur University, Kolkata-700032, West Bengal,India.}
\email{pproy.math.rs@jadavpuruniversity.in}

\subjclass[{AMS} Subject Classification:]{Primary 30A10, 30H05, 30C35, Secondary 30C45}

\def\thefootnote{}
\footnotetext{ {\tiny File:~\jobname.tex,
printed: \number\year-\number\month-\number\day,
          \thehours.\ifnum\theminutes<10{0}\fi\theminutes }
} \makeatletter\def\thefootnote{\@arabic\c@footnote}\makeatother

\begin{abstract} 
In this article, we analyze refined and improved versions of the classical Bohr inequality for the function class $\mathcal{B}$, which consists of analytic self- mappings defined on the unit disk $\mathbb{D}$. We improve the Bohr-Rogosinski inequality for the class $\mathcal{B}$ of analytic self-maps by incorporating the area measure of sub-disks $\mathbb{D}_r$ in $\mathbb{D}$. 
\end{abstract}

\maketitle
\pagestyle{myheadings}
\markboth{M. B. Ahamed and P. P. Roy}{Bohr-Rogosinski Inequalities for bounded analytic functions}

\section{Introduction}
Let $ \mathbb{D}:=\{z\in \mathbb{C}:|z|<1\} $ denote the open unit disk in $ \mathbb{C} $. Harald Bohr in $1914$ (see \cite{Bohr-1914}) proved that if $H_\infty$ denotes the class of all bounded analytic functions $f$ on $\mathbb{D}$, then the following inequality holds
\begin{align}\label{Eq-1.1}
	\mathcal{B}_0(f,r):=|a_0|+\sum_{n=1}^{\infty}|a_n|r^n\leq||f||_\infty:=\displaystyle\sup_{z\in\mathbb{D}}|f(z)| \;\;\mbox{for}\;\; 0\leq r\leq\frac{1}{3},
\end{align}
where $a_k=f^{(k)}(0)/k!$ for $k\geq 0$. The constant $1/3$ is called the Bohr radius and the inequality in \eqref{Eq-1.1} is called Bohr inequality for the class $\mathcal{B}$ of analytic self-maps on $\mathbb{D}$. Henceforth, if there exists a positive real number $r_0$ such that an inequality of the form \eqref{Eq-1.1} holds for every elements of a class $\mathcal{M}$ for $0\leq r\leq r_0$ and fails when $r>r_0$, then we shall say that $r_0$ is an sharp bound for $r$ in the inequality w.r.t. to class $\mathcal{M}$.\vspace{1.2mm}

 The constant $ 1/3 $ and the inequality \eqref{Eq-1.1} are called respectively, the Bohr radius and the Bohr's inequality, for the class $\mathcal{B}$. Moreover, for $ f_a(z):={(a-z)}/{(1-az)},$ where $ a\in [0,1) $, it follows easily that $ \mathcal{B}_0(f_a, r)>||f||_{\infty} $ if, and only if, $ r>1/(1+2a) $, which shows that  the constant $ 1/3 $ is best possible as the limiting case $ a\rightarrow 1^{-} $ suggests. Several other proofs of this interesting inequality were given in different articles  (see, e.g. \cite{Paulsen-PLMS-2002,Tomic-1962,Sidon-1927}). \vspace{1.5mm}

In the majorant series $ \mathcal{B}_0(f,r)$ of the function $ f\in\mathcal{B} $, the beginning terms play some significant in the related discussion about the Bohr inequality. For instance, in the case of  $|a_0|=0$, Tomic \cite{Tomic-1962} has proved the inequality \eqref{Eq-1.1}  for $0\leq r \leq {1}/{2}$ and if the term $|a_0|$ is replaced by  $|a_0|^2$, then the constant $1/3$ can be replaced by $1/2$. Moreover, if  the term $|a_0|$ in $ \mathcal{B}_0(f,r) $ is replaced by $|f(z)|$, then the constant $1/3$ can be replaced by sharp constant $\sqrt{5}-2$ established in \cite{Ponnusamy-2017} (see also \cite{Kayu-Kham-Ponnu-2021-JMAA}). \vspace{1.2mm} 

Bohr’s theorem received greater interest after it was used by Dixon \cite{Dixon & BLMS & 1995} to characterize Banach algebras that satisfy von Neumann inequality. The generalization of Bohr’s theorem is now an active area of research: Aizenberg \emph{et al.} \cite{Aizn-PAMS-2000,Aizeberg-PLMS-2001,Aizenberg-SM-2005,Aigenber-CMFT-2009}, and Aytuna and Djakov \cite{Ayt & Dja & BLMS & 2013} studied the Bohr property of bases for holomorphic functions; Ali \emph{et al.} \cite{Ali & Ng & CVEE & 2016} found the Bohr radius for the class of starlike logharmonic mappings; while Paulsen and Singh \cite{Paulsen-PLMS-2002,Paulsen-PAMS-2004,Paulsen-BLMS-2006} extended the Bohr inequality to Banach algebras. For detailed information regarding the development of the Bohr inequality, the readers are referred to the survey articles \cite{Abu Muhanna-Survey} and \cite{Ponnusamy-Vijayakumar}. For the recent developments, we refer to the articles \cite{Ahamed-AASFM-2022,Ali-RM-2019,Alkhaleefah-PAMS-2019,Bhowmik-CMFT-2019, Liu-JMAA-2021, Das-JMAA-2022,Kay & Pon & AASFM & 2019,Kumar-Sahoo-MJM-2021,Lata-Singh-PAMS-2022,Ponnusamy-CMFT-2020,Ponnusamy-JMAA-2022, Ponnusamy-RM-2021, Allu-CMB-2022} and references therein.  \vspace{1.2mm}

Before proceeding with the discussion, let us introduce the necessary notations.
\subsection{\bf Basic Notations}  Let 
\begin{align*}
	\mathcal{B}=\{f\in H_{\infty} \;:\;||f||_{\infty}\leq 1\}.
\end{align*} Also, for $f(z)=\sum_{n=0}^{\infty}a_nz^n\in \mathcal{B}$ and $f_0(z):=f(z)-f(0)$, we let for convenience 
\begin{align*}
	B_k(f,r):=\sum_{n=k}^{\infty}|a_n|r^n \;\mbox{for }k\geq 0,\; ||f_0||^2_r:=\sum_{n=1}^{\infty}|a_n|^2r^{2n} \;\mbox{and} \;||f_1||^2_r:=\sum_{n=2}^{\infty}|a_n|^2r^{2n}
\end{align*}
Further, we define the quantity $A(f_0,r)$ by 
\begin{align*}
		A(f_0,r):=\left(\dfrac{1}{1+|a_0|}+\dfrac{r}{1-r}\right)||f_0||^2_r.
\end{align*}
The concept of the Rogosinski radius, like that of the Bohr radius, was introduced in \cite{Rogosinski-1923}. However, the Rogosinski radius has not been studied as extensively as the Bohr radius (see \cite{Landau-Gaier-1986,Schur-Szego-1925}). The definition of the Rogosinski radius is provided below.
 \begin{thmA}(Rogosinski Theorem)\label{th-1.2}
 	If $ g(z)=\sum_{n=0}^{\infty}b_nz^n \in \mathcal{B},$  then for every $ N\geq 1 $, 
 	\begin{align}\label{e-0.1}
 		\bigg|\sum_{n=0}^{N-1}b_nz^n\bigg|\leq 1\; \mbox{for}\; |z|\leq \frac{1}{2}.
 	\end{align}
 	The radius $ 1/2 $ is best possible and is called Rogosinski radius.
 \end{thmA}
Inspired by the articles of \cite{Ponnusamy-2017} and \cite{Kayu-Kham-Ponnu-2021-JMAA}, the Bohr-Rogosinski sum, $R_n^f(z)$, for a function $f \in \mathcal{B}$ with Taylor series $f(z) = \sum_{n=0}^{\infty} a_n z^n$, is defined by
 \begin{align}\label{e-0.2}
 	R^f_n(z):=|f(z)|+\sum_{n=N}^{\infty}|a_n|r^n, \; |z|=r.
 \end{align}
 An interesting observation is that for $N=1$, the quantity defined in $\eqref{e-0.2}$ is related to the classical Bohr sum where $|f(0)|$ is replaced by $|f(z)|$. The inequality $R_n^f(z) \leq 1$ is known as the Bohr-Rogosinski inequality. If $B$ and $R$ denote the Bohr radius and the Bohr–Rogosinski radius, respectively, it is easy to see that $B = 1/3 < 1/2 = R$. For recent developments concerning the Bohr-Rogosinski phenomenon, the reader is referred to the articles by Das \cite{Das-JMAA-2022}, Kayumov \textit{et al.} \cite{Kayu-Kham-Ponnu-2021-JMAA}, and the references cited therein.\vspace{1.2mm}
 
 It is worth mentioning that various questions concerning the related new concept, the Bohr-Rogosinski phenomenon, including its refined forms, are currently being studied (see, e.g., \cite{Aizenberg-AMP-2012,Huang-Liu-Pon-AMP-2020,Liu-Liu-Ponnusamy-2021}). However, to the best of our knowledge, unlike the refined versions of the Bohr inequality for different classes of functions, there is no strengthened version of the Rogosinski inequality $\eqref{e-0.1}$ that holds for $r \leq 1/2$ and for all $N \in \mathbb{N}$. Motivated by this gap, we refine the Rogosinski inequality for the classes of functions $\mathcal{B}$ and the subordination class. We also improve a certain refined Rogosinski inequality in terms of the quantity $S_r$, which is the planar integral of an analytic function defined on $\mathbb{D}$. Consequently, in this article, the Bohr–Rogosinski inequality is considered to be analogous to a Bohr-type inequality, which is a variant of the inequality found in $\eqref{Eq-1.1}$ or $\eqref{e-0.2}$. \vspace{1.2mm}

  For a function $f(z) = \sum_{n=0}^{\infty} a_n z^n$, we define $f_0(z) := f(z) - f(0)$ for convenience. Throughout this paper, we will use the following notations
  \begin{align*}
  	{||f_0||}^2:=\sum_{n=1}^{\infty}|a_n|^2r^{2n}\;\;\mbox{and}\;\; B_N(f,r):=\sum_{n=N}^{\infty}|a_n|r^n \;\; \mbox{for}\;\; N\in \mathbb{N}\cup\{0\}. 
  \end{align*}
 Using the foundation of the Rogosinski inequality and the Rogosinski radius (first studied in \cite{Rogosinski-1923}), Kayumov \textit{et al.} \cite{Kayu-Kham-Ponnu-2021-JMAA} (and \cite{Ponnusamy-2017}) obtained the Bohr-Rogosinski inequality and Bohr-Rogosinski radius for the class $\mathcal{B}$.
\begin{thmB}\cite[Corollary 1]{Kayu-Kham-Ponnu-2021-JMAA}\label{Cor-1.1}
	Suppose that $ f(z)=\sum_{n=0}^{\infty}a_nz^n\in\mathcal{B} $. Then for $N\in \mathbb{N},$
	\begin{align*}
		|f(z)|+B_N(f,r)\leq 1\;\;\mbox{for}\;\; r\leq R_N,
	\end{align*}
	where $ R_N $ is the positive root of the equation $ 2(1+r)r^N-(1-r)^2=0 $. The radius $ R_N $ is the best possible. Moreover, 
	\begin{align*}
		|f(z)|^2+B_N(f,r)\leq 1\;\;\mbox{for}\;\; r\leq R^{\prime}_N,
	\end{align*}
	where $ R^{\prime}_N $ is the positive root of the equation $ (1+r)r^N-(1-r)^2=0 $. The radius $ R^{\prime}_N $ is the best possible.
\end{thmB}
Following the initiation of work by Kayumov and Ponnusamy \cite{Kayumov-Ponnusamy-2018-CR Academy}, Liu et al. \cite{Liu-Shang-Xu-2018} considered several Bohr-type inequalities for the family $\mathcal{B}$, where the Taylor coefficients of the classical Bohr inequality were partly or completely replaced by higher-order derivatives of $f$. We recall only one such result here.
\begin{thmC}\cite[Theorem 2.1]{Liu-Shang-Xu-2018}
	Suppose that $f\in\mathcal{B}$ and $f(z)=\sum_{n=0}^{\infty} a_nz^n$. Then the following sharp inequality holds:
	\begin{align*}
		|f(z)|+|f^{\prime}(z)|r+\sum_{n=2}^{\infty}|a_n|r^n\leq1\;\;\mbox{for}\;\;|z|=r\leq\frac{\sqrt{17}-3}{4}.
		\end{align*}
\end{thmC}
In $2021$, Liu et al. \cite{Liu-Liu-Ponnusamy-2021} presented an improved version of the above theorem.
\begin{thmD}\cite[Theorem 7]{Liu-Liu-Ponnusamy-2021}
	Suppose that $f\in\mathcal{B}$ and $f(z)=\sum_{n=0}^{\infty}a_nz^n$. Then 
	\begin{align*}
		|f(z)|+|f^{\prime}(z)|r+\sum_{n=2}^{\infty}|a_n|r^n+A(f_0,r)\leq1
	\end{align*}
	for $|z|=r\leq\frac{\sqrt{17}-3}{4}$ and the constant $\frac{\sqrt{17}-3}{4}$ is best possible.
\end{thmD}
In a recent work, Ponnusamy and Vijayakumar \cite{Ponnusamy-Vijayakumar} presented an extension of Theorem B that applies to harmonic quasiconformal maps, under which Theorem B and Corollary A are obtained as special cases.\vspace{1.2mm}

However, Bohr-Rogosinski inequalities have been the focus of substantial research attention in recent times. Extensive investigations have been carried out on Theorem B and Corollary A to derive their generalized or strengthened versions for different classes of functions. For instance, these results have been established for:\vspace{1.2mm}

$\bullet$ Certain classes of harmonic mappings (see, e.g., \cite{Ahamed-RMJM-2021,Ahamed-AMP-2021,Ahamed-CVEE-2021,Allu-BSM-2021}).

$\bullet$ Classes of functions in one and higher dimensions (see, e.g., \cite{Boas-Khavinson_PAMS-1997,Liu-Liu-Ponnusamy-2021,Liu-Ponnusamy-PAMS-2021}).

$\bullet$ Classes of operator-valued functions in multidimensional settings (see, e.g., \cite{Allu-CMB-2022}).\vspace{1.2mm}

Furthermore, Liu \cite{Liu-JMAA-2021}, and Allu and Arora \cite{Allu-Arora-2022} have recently developed the Bohr-Rogosinski inequality and generalized the notion of the Bohr-Rogosinski phenomenon in terms of Schwarz functions. \vspace{1.2mm}

 Besides the study of sharp refined Bohr inequality, improved versions have also been established in recent years. However, to study the improved Bohr radius, the majorant series in the classical Bohr inequality plays a significant role, along with the quantity $S_r$ (and its integral powers) and its analogues for certain classes of harmonic mappings or operator-valued functions. A number of results in this area have been established in the last couple of years following the publication of the paper by Kayumov and Ponnusamy \cite{Kayumov-CRACAD-2018}. For some other aspects of the improved Bohr inequalities, readers may refer to the articles \cite{Ahamed-CVEE-2021,Allu-CMB-2022} and the references cited therein.
  \subsection{Improved Bohr inequalities for the class $\mathcal{B}$.}
 Let $f$ be holomorphic in $\mathbb{D}$, and for $0<r<1$,  let $\mathbb{D}_r:=\{z\in \mathbb C: |z|<r\}$.
Throughout the paper,  $S_r=S_r(f)$ denotes the planar integral
\begin{align*}
	S_r=\int_{\mathbb D_r} |f'(z)|^2 d A(z).
\end{align*}
Note that if  $f(z)=\sum_{n=0}^\infty a_nz^n$, then $S_r=\pi \sum_{n=1}^\infty n|a_n|^2 r^{2n}.$  If $f$ is a univalent function, then $S_r$ is the area of  $f(\mathbb D_r)$.\vspace{1.2mm}

In $2018$, Kayumov and Ponnusamy \cite{Kayumov-CRACAD-2018} obtained the following result, which is an improved version of the Bohr inequality derived using the sharp bounds of the quantity $S_r$. Let us now recall a couple of recent results for our reference.
\begin{thmE}\cite[Theorem 1]{Kayumov-CRACAD-2018}\label{th-1.12}
	Suppose that  $ f(z)=\sum_{n=0}^{\infty}a_nz^n\in\mathcal{B} .$ Then 
	\begin{enumerate}
		\item[(i)] $ B_0(f,r)+ \frac{16}{9} \left(\frac{S_{r}}{\pi}\right) \leq 1\; \mbox{for}\; r \leq \frac{1}{3}, $
		and the numbers $1/3$, $16/9$ cannot be improved. Moreover, 		
		\item[(ii)] $ |a_{0}|^{2}+B_1(f,r)+ \frac{9}{8} \left(\frac{S_{r}}{\pi}\right) \leq 1\;\mbox{for}\; r \leq \frac{1}{2}, $ and the numbers $1/2$, $9/8$ cannot be improved.
	\end{enumerate}
\end{thmE}
In $ 2020 $, Ismagilov \emph{et al.} \cite{Ismagilov-2020-JMAA} continued  investigating on Theorem \ref{th-1.12} further and proved the following sharp results in view of addition of two degree polynomial expression of the quantity $ S_r/\pi $ with some suitable coefficients.
\begin{thmF}\cite[Theorem 1, Theorem 2]{Ismagilov-2020-JMAA} \label{th-1.10}
	Suppose that  $ f(z)=\sum_{n=0}^{\infty}a_nz^n\in\mathcal{B} $. Then 
	\begin{enumerate}
		\item[(i)] $B_0(f,r)+\frac{16}{9}\left(\frac{S_r}{\pi}\right)+\lambda_1\left(\frac{S_r}{\pi}\right)^2\leq 1\;\; \mbox{for}\;\; r\leq\frac{1}{3},$
	where $\lambda_1$ is given by $\lambda_1=\frac{4(486-261a-324a^2+2a^3+30a^4+3a^5)}{81(1+a)^3(3-5a)}=18.6095...$
	and $ a\approx 0.567284 $, is the unique root of the equation $ -405+473t+402t^2+38t^3+3t^4+t^5=0 $ in the interval $ (0,1) .$ Moreover,
	\item[(ii)] $|f(z)|^2+B_1(f,r)+\frac{16}{9}\left(\frac{S_r}{\pi}\right)+\lambda_2\left(\frac{S_r}{\pi}\right)^2\leq 1\;\; \mbox{for}\;\; r\leq\frac{1}{3},$
	where $\lambda_2=\frac{-81+1044a+54a^2-116a^3-5a^4}{162(a+1)^2(2a-1)}=16.4618..$
	and $ a\approx 0.537869 $, is the unique root of the equation $ -513+910t+80t^2+2t^3+t^4=0 $	in the interval $ (0,1) $.
\end{enumerate}
	 The equalities are achieved for the function
	\begin{align}\label{Eq-2.2}
		f_a(z) =\frac{a-z}{1-az},\;\; a\in[0,1).
	\end{align}
\end{thmF}
Subsequently, Huang \textit{et al.} \cite{Huang-Liu-Ponnu-CVEE-2021} presented a further improvement of the inequality (i) of Theorem F by considering the addition of the square of $S_r/\pi$, which yielded the following sharp inequality.
\begin{align*}
	B_0(f,r)+\left(\frac{1}{1+|a_0|}+\frac{r}{1-r}\right){||f_0||}^2+\frac{8}{9}\left(\frac{S_r}{\pi}\right)+\lambda\left(\dfrac{S_r}{\pi}\right)^2\leq 1 \;\; \mbox{for}\;\; r\leq\dfrac{1}{3},
\end{align*} which is an  improved version of the classical Bohr inequality,	where $ \lambda=14.796883... $, and showed that the equality  is achieved for the function $f_a$  defined by \eqref{Eq-2.2}. The authors \cite{Huang-Liu-Ponnu-CVEE-2021} have also established the following sharp inequality
\begin{align*}
	|a_0|^2+B_1(f,r)+\left(\frac{1}{1+|a_0|}+\frac{r}{1-r}\right){||f_0||}^2+\frac{9}{8}\left(\frac{S_r}{\pi}\right)+\lambda\left(\dfrac{S_r}{\pi}\right)^2\leq 1
\end{align*}
for $ r\leq{1}/(3-|a_0|) $,	where $ \lambda=13.966088... $ and proved that the equality can be achieved for the function $ f_a$ defined by \eqref{Eq-2.2}.\vspace{1.2mm}
We now turn our attention to a different setting of the quantity $ S_r $. Recently, Ismagilov \emph{et al.} \cite{Ismagilov-2021-JMS} observed the following fact
\begin{align}\label{e-11.18}
	\dfrac{S_r}{\pi-S_r}\leq \dfrac{r^2(1-|a_0|^2)^2}{(1-r^2)(1-r^2|a_0|^4)}.
\end{align}
It will be an interesting question to find sharp improved Bohr-type inequalities or Bohr-Rogosinski inequalities in terms of the addition of the quantity $ S_r/(\pi-S_r) $. However, with the help of this new setting, Ismagilov \emph{et al.} \cite{Ismagilov-2021-JMS} initiated the study of improved Bohr-type inequalities by reformulating the inequality in Theorem \ref{th-3.2} replaced the quantity $S_r/\pi$ by  $S_r/(\pi-S_r)$ and obtained the following sharp result.
\begin{thmG} \label{th-1.18} \cite{Ismagilov-2021-JMS}
	Suppose that $ f(z)=\sum_{n=0}^{\infty}a_nz^n\in\mathcal{B} $. Then 
	\begin{align*}
		B_0(f,r)+ \frac{16}{9} \left(\frac{S_{r}}{\pi-S_r}\right) \leq 1 \quad \mbox{for} \quad r \leq \frac{1}{3},
	\end{align*}
	and the number $16/9$ cannot be improved. Moreover, 
	\begin{align*}
		|a_{0}|^{2}+B_1(f,r)+ \frac{9}{8} \left(\frac{S_{r}}{\pi-S_r}\right) \leq 1 \quad \mbox{for} \quad r \leq \frac{1}{2},
	\end{align*}
	and the number $9/8$ cannot be improved.
\end{thmG}
Inspired by Theorem F in view of Theorem G , we continue the study with the quantity $S_r/(\pi-S_r)$ further and proved the following sharp result.

\begin{prob}\label{Q-2.2}
	Whether we can derive sharp version of Theorem D in the setting of Theorem F and Theorem G with the additional two non-negative term $S_r/(\pi-S_r)$ instead of $S_r/\pi$ and without decreasing the radius?
\end{prob}
We observe that while there are numerous sharp results concerning improved Bohr inequalities for the class $\mathcal{B}$, the literature contains only a very few results on improved Bohr-Rogosinski inequalities (see \cite{Liu-Liu-Ponnusamy-2021}). This gap motivates us to propose the following problem for further study.
\begin{prob}\label{P-1}
	Can we establish an improved Bohr-Rogosinski inequality for the class $\mathcal{B}$?
\end{prob}
In view of Theorems D, F and G, it is natural to raise the following.
\begin{prob}\label{Q-2.1}
	Can we derive a sharp version of Theorem D within the framework of Theorem F and Theorem G, while incorporating the two additional non-negative terms and maintaining the original radius?
\end{prob}
The main objective of this paper is to establish an improved version of the Bohr-Rogosinski inequalities, thus addressing Problem \ref{Q-2.2}, Problem \ref{P-1} and Problem \ref{Q-2.1}. The paper is organized as follows. Section 2 presents our main conclusions. Theorems \ref{th-3.2} and \ref{th-3.3} provide successful answers to Problem \ref{Q-2.1} and Problem \ref{Q-2.2}, which relate to the definition of the Bohr-type operator applicable to analytic functions and their sections.
\section{\bf Improved Bohr-Rogosinski inequalities for the class $ \mathcal{B} $}
In the study of Bohr's inequality, it is well-known that the sharp bounds of the coefficients $ a_n $ in the majorant series and in case of Rogosinski inequality, the sharp bounds of $ |f(z)| $ have the key roles. However,  the sharp bounds of the quantity $ S_r $ for the class of functions $f\in \mathcal{B} $  is established by Kayumov and Ponnusamy \cite{Ponnusamy-2017} which is
\begin{align}\label{ee-2.1}
	\frac{S_r}{\pi}=\sum_{n=1}^{\infty}n|a_n|^2r^{2n}\leq \frac{r^2(1-|a_0|^2)^2}{(1-|a_0|^2r^2)^2} \,\,\,\,\mbox{for}\,\,\,\, 0<r\leq1/\sqrt{2}.
\end{align}
In Theorem \ref{th-3.2}, we present an affirmative answer to this Problem \ref{Q-2.1}. 
\begin{thm}\label{th-3.2}
		Suppose that $ f(z)=\sum_{n=0}^{\infty}a_nz^n\in\mathcal{B} $. Then 
	\begin{align*}
		\mathcal{A}_{f}(r):=|f(z)|+|f^{\prime}(z)|r+B_2(f,r)+A(f_0,r)+\frac{221-43\sqrt{17}}{64}\left(\frac{S_r}{\pi}\right)+\lambda\left(\dfrac{S_r}{\pi}\right)^2\leq 1
	\end{align*}
	where 
	\begin{equation*}
		\lambda=\dfrac{A_1(a_*)}{128(1+a_*)^3A_2(a_*)}=18.0215\dots
	\end{equation*}where
	\begin{align*}
		A_1(a_*)=&128(-6311+1521\sqrt{17})+32(-459367+111409\sqrt{17})a\\&\quad+48(368037-89251\sqrt{17})a^2+16(6825181-1655339\sqrt{17})a^3\\&\quad+1280(-44509+10795\sqrt{17})a^4+6(-46228497+11212055\sqrt{17})a^5\\&\quad+7(-2654823+643889\sqrt{17})a^6+32(7818303-1896217\sqrt{17})a^7\\&\quad+36(4165553-1010295\sqrt{17})a^8
	\end{align*}and
	\begin{align*}
		A_2(a_*)=&(2008-488\sqrt{17})+(33442-8110\sqrt{17})a+3(-36847+8937\sqrt{17})a^2\\&\quad+(79457-1927\sqrt{17})a^3
	\end{align*}
	and $a_*\approx0.600976$ is the unique root in the interval $(0,1)$ of the equation  $\Psi_1^*(t)=0$, where 
	\begin{align}\label{eq-2.2}
		\Psi_1^*(t)&:=2048(-169479+41105\sqrt{17})+64(-74661865+18108159\sqrt{17})t\\&-64(-470704167+114162545\sqrt{17})t^2+\nonumber16(-698169901+169331035\sqrt{17})t^3\\&+16(-9949839967+2413190649\sqrt{17})t^4-172(-914721543+221852561\sqrt{17})t^5\nonumber\\&+(298599878708-72421108204\sqrt{17})t^6+(-277864274087+67391985393\sqrt{17})t^7\nonumber\\&+(-376869563493+91404295139\sqrt{17})t^8\nonumber-4(-40230249983+9757268825\sqrt{17})t^9\\&+(363756131186-88223820638\sqrt{17})t^{10}+2(-35641396573+8644308395\sqrt{17})t^{11}\nonumber\\&+(-94566008398+22935625954\sqrt{17})t^{12}\nonumber.
	\end{align} 
	The equality  is achieved for the function 
	$f_a$ defined by \eqref{Eq-2.2}.
\end{thm}
\begin{thm}\label{th-3.3}
		Suppose that $ f(z)=\sum_{n=0}^{\infty}a_nz^n\in\mathcal{B} $. Then 
	\begin{align*}
		\mathcal{J}_{f}(r)&:=|f(z)|+|f^\prime(z)|r+B_2(f,r)+A(f_0,r)+\frac{(221-43\sqrt{17})}{64}\left(\frac{S_r}{\pi-S_r}\right)\\&\quad+\mu\left(\frac{S_r}{\pi-S_r}\right)^2\leq 1 \;\; \mbox{for}\;\; r\leq\dfrac{\sqrt{17}-3}{4},
	\end{align*}
	where 
	\begin{equation*}
		\mu=\dfrac{E_1(a_{**})}{128(1+a_{**})^3E_2(a_{**})}=16.0824\dots
	\end{equation*}where
	\begin{align*}
		E_1(a_{**})=&(321568-77024\sqrt{17})+8(-1166421+282995\sqrt{17})a_{**}\\&+(2287452-554244\sqrt{17})a_{**}^2+16(-974531+236341\sqrt{17})a_{**}^3\\&-40(-181177+43951\sqrt{17})a_{**}^4+(88788654-21534594\sqrt{17})a_{**}^5\\&+7(-2416805+586147\sqrt{17})a_{**}^6+32(-3664887+888865\sqrt{17})a_{**}^7\\&\quad+36(-1952633+473583\sqrt{17})a_{**}^8	\end{align*}and
		\begin{align*}
			E_2(a_{**})=&8(-251+61\sqrt{17})+(-33442+8110\sqrt{17})a_{**}-3(-36847+8937\sqrt{17})a_{**}^2\\&\quad+7(-11351+2753\sqrt{17})a_{**}^3
		\end{align*}
	and $a_{**}\approx 0.565671$ is the unique  root in  $(0,1)$ of the equation  $\Psi^*_2(t)=0$, where
	\begin{align}\label{e-3.2}
	\Psi^*_2(t)&:=128(-12782703+3100265\sqrt{17})+64(-207084487+50225361\sqrt{17})t\\&\quad-32(-129741581\nonumber+314668555\sqrt{17})t^2+(32440299472-7867928752\sqrt{17})t^3\\&\quad+16(-10605671261\nonumber+2572253099\sqrt{17})t^4+(85784031284-20805683692\sqrt{17})t^5\\&\quad+(197798732932\nonumber-47973239324\sqrt{17})t^6+(-193832516055+47011290433\sqrt{17})t^7\\&\quad+(-132672807253\nonumber+32177882227\sqrt{17})t^8+(75432873692-18295159172\sqrt{17})t^9\\&\quad+(170513297282\nonumber-41355549134\sqrt{17})t^{10}+(-33414320906+8104163206\sqrt{17})t^{11}\\&\quad+2(-22164247583\nonumber+5375619641\sqrt{17})t^{12}.
	\end{align} 

	The equality  is achieved for the function 
	$f_a$ defined by \eqref{Eq-2.2}.
\end{thm}
 To prove Theorems \ref{th-3.2} and \ref{th-3.3}, we take the help of the following lemma established by Liu \emph{et al.} \cite{Liu-Liu-Ponnusamy-2021}.
\begin{lemA}\cite{Liu-Liu-Ponnusamy-2021}\label{lem-44.1}
	Suppose that $ f(z)=\sum_{n=0}^{\infty}a_nz^n\in\mathcal{B} $. Then for any N$\in\mathbb{N}$, the following inequality holds:
	\begin{align*}
		B_N(f,r)+sgn(t)\sum_{n=1}^{t}|a_n|^2\dfrac{r^N}{1-r}+\left(\dfrac{1}{1+|a_0|}+\dfrac{r}{1-r}\right)\sum_{n=t+1}^{\infty}|a_n|^2r^{2n}\leq \dfrac{(1-|a_0|^2)r^N}{1-r}, 
	\end{align*}
	\;\;\mbox{for}\;\; $ r\in[0,1),$ where $t=\lfloor{(N-1)/2}\rfloor.$
\end{lemA}

\begin{proof}[\bf Proof of Theorem \ref{th-3.2}]
	By assumption $ f(z)=\sum_{n=0}^{\infty}a_nz^n\in\mathcal{B}.$ Since $f(0)=a_0,$ by the consequence of Schwarz-Pick Lemma \cite{Ruscheweyh-1985} we can prove that for $z=re^{i\theta}\in \mathbb{D},$
	
	\begin{align}\label{e-5.1}
		|f(z)|\leq \dfrac{r+|a_0|}{1+r|a_0|}\;\;\mbox{for}\;\; |z|\leq r.
	\end{align}
	
Let $|a_0|=a\in(0,1)$. In view of \eqref{ee-2.1}, \eqref{e-5.1}  and by Lemma \ref{lem-44.1} with $N=1$, 	we obtain
\begin{align*}
	\mathcal{A}_{f}(r)& \leq \frac{r+a}{1+ra}+\frac{r}{1-r^2}\left(1-\left(\frac{r+a}{1+ra}\right)^2\right)+\frac{(1-a^2)r^2}{1-r}\\&\quad+\frac{221-43\sqrt{17}}{64}\frac{(1-a^2)^2r^2}{(1-a^2r^2)^2}+\lambda\left(\frac{(1-a^2)^2r^2}{(1-a^2r^2)^2}\right)^2\\&=\frac{r+a}{1+ra}+\frac{r(1-a^2)}{(1+ar)^2}+\frac{(1-a^2)r^2}{1-r}+\frac{221-43\sqrt{17}}{64}\frac{(1-a^2)^2r^2}{(1-a^2r^2)^2}\\&+\lambda\left(\frac{(1-a^2)^2r^2}{(1-a^2r^2)^2}\right)^2:=\mathcal{A}^*_{f}(r) 
\end{align*}
Since $\mathcal{A}^*_{f}(r)$ is an increasing function of $r$, we have for $r\leq(\sqrt{17}-3)/4$,
\begin{align*}
	\mathcal{A}^*_{f}(r)&\leq\mathcal{A}^*_{f}((\sqrt{17}-3)/4)\\&\quad=\frac{(-3+\sqrt{17})+4a}{4+\left(-3+\sqrt{17}\right)a}+\frac{(-3+\sqrt{17})^2(1-a^2)}{4(7-\sqrt{17})}+\frac{4(-3+\sqrt{17})(1-a^2)}{\left(4+(-3+\sqrt{17})|a_0|\right)^2}\\&\quad+\frac{(-3+\sqrt{17})^2(221-43\sqrt{17})(1-a^2)^2}{16\left(8+(-13+3\sqrt{17})a^2\right)^2}+\lambda\frac{16(-3+\sqrt{17})^4(1-a^2)^4}{\left(8+(-13+3\sqrt{17})a^2\right)^4}\\&=1-\dfrac{8(1-a)^3A_3(a)}{(7-\sqrt{17})(4+(-3+\sqrt{17
		})a)^2(8+(-13+3\sqrt{17})a^2)^4},
\end{align*}
where \begin{align}\label{e-5.4}
	A_3(a):=&64(-11563+2829\sqrt{17})-128(-6311+1521\sqrt{17})a\\&\quad+\nonumber(7349872-1782544\sqrt{17})a^2+16(-368037+89251\sqrt{17})a^3\\&\quad\nonumber+4(-6825181+1655339\sqrt{17})a^4-256(-44509+10795\sqrt{17})a^5\\&\quad+(46228497-11212055\sqrt{17})a^6+(2654823-643889\sqrt{17})a^7\nonumber\\&\quad+4(-7818303+1896217\sqrt{17})a^8+\nonumber4(-4165553+1010295\sqrt{17})a^9\\&\quad\nonumber-128\lambda(1-a)(1+a)^4\times L_a,\nonumber
\end{align} 
where $L_a$ is defined as 
\begin{align*}
	L_a=(3580-868\sqrt{17})+4(-3187+773\sqrt{17})a+(11351-2753\sqrt{17})a^2.
\end{align*}
Next, we verify that the function $A_3(a)$ has exactly one stationary point $a_*\approx 0.600976$ in $(0,1), $ which is the unique root in $(0,1)$ of the equation $A_3(a)=0$. We remark that $A_3^{\prime}(a)=0$ is equivalent to 
\begin{align*}
	&128(1+a)^3\lambda  A_2(a)=A_1(a),
\end{align*}
 
from which we obtain the value of $\lambda$ mentioned in the statement of the theorem. By a computation using Mathematica, we obtain that the number $a_*\approx 0.600976$ is the unique root in $(0,1)$ of the equation $A_3^{\prime}(a)=0$, where 
\begin{align*}
	A_3^{\prime}(a)=&A_1(a)-128\times18.0215\times(1+a)^3[(2008-488\sqrt{17})+(33442-8110\sqrt{17})a\\&\quad+3(-36847+8937\sqrt{17})a^2+(79457-1927\sqrt{17})a^3].
\end{align*} The next step involves substituting the value of $\lambda$ into the expression of $A_3(a_*)$, resulting the following
\begin{align*}
&A_3(a_*)=\dfrac{8\Psi^{*}_1(a_*)}{A_2(a_{**})},	
\end{align*}
where $\Psi_1^*(t)$ is given by \eqref{eq-2.2}.
It can be easily seen that $a_*$ is the unique root in $(0,1) $ of the equation  $\Psi_1^*(t)=0.$ Thus  $A_3(a_*)=0$ and  $A_3^{\prime}(a_*)=0$. Besides this observation, we have $A_3(0)=5052.19>0$ and $A_3(1)=10949.3>0.$ Consequently, $A_3(a)\geq0 $ in the interval $(0,1) $, which proves that $\mathcal{A}_{f}(r)\leq 1$ for $r\leq(\sqrt{17}-3)/4.$  
Finally, to prove that the constant $\lambda$ is sharp, we consider the function $f_a$ is defined by \eqref{Eq-2.2}
for some value of $a$ in $[0,1).$\vspace{1.2mm}

 For $f_a$, a straightforward calculation shows that 
\begin{align*}
	\mathcal{A}_{f_a}(r)&=|f(z)|+|f^{\prime}(z)|r+B_2(f,r)+\left(\frac{1}{1+|a_0|}+\frac{r}{1-r}\right){||f_0||}^2\\&\quad+\frac{221-43\sqrt{17}}{64}\left(\frac{S_r}{\pi}\right)+\lambda_1\left(\dfrac{S_r}{\pi}\right)^2\\&=\left(\dfrac{a+r}{1+ra}\right)+\dfrac{(1-a^2)r}{(1+ar)^2}+\dfrac{a(1-a^2)r^2}{(1-ar)}+\dfrac{(1+ar)(1-a^2)^2r^2}{(1+a)(1-r)(1-r^2a^2)}\\&\quad+\frac{221-43\sqrt{17}}{64}\dfrac{(1-a^2)^2r^2}{(1-r^2a^2)^2}+\lambda_1\left(\dfrac{(1-a^2)^2r^2}{(1-a^2r^2)^2}\right)^2\\&=\left(\dfrac{a+r}{1+ra}\right)+\dfrac{(1-a^2)r}{(1+ar)^2}+\dfrac{(1-a^2)r^2}{1-r}+\frac{221-43\sqrt{17}}{64}\dfrac{(1-a^2)^2r^2}{(1-r^2a^2)^2}\\&\quad+\lambda_1\left(\dfrac{(1-a^2)^2r^2}{(1-a^2r^2)^2}\right)^2
\end{align*} 
For $r=(\sqrt{17}-3)/4$, the last expression becomes 
\begin{align*}
		\mathcal{A}_{f_a}((\sqrt{17}-3)/4)&=\frac{(-3+\sqrt{17})+4a}{4+\left(-3+\sqrt{17}\right)a}+\frac{(-3+\sqrt{17})^2(1-a^2)}{4(7-\sqrt{17})}\\&\quad+\frac{4(-3+\sqrt{17})(1-a^2)}{\left(4+(-3+\sqrt{17})|a_0|\right)^2}+\frac{(-3+\sqrt{17})^2(221-43\sqrt{17})(1-a^2)^2}{16\left(8+(-13+3\sqrt{17})a^2\right)^2}\\&\quad+16(-3+\sqrt{17})^4\lambda\times\frac{(1-a^2)^4}{\left(8+(-13+3\sqrt{17})a^2\right)^4}\\&\quad+(\lambda_1-\lambda)\frac{16(-3+\sqrt{17})^4(1-a^2)^4}{\left(8+(-13+3\sqrt{17})a^2\right)^4}.
\end{align*} 

Choose $a$ as the unique root $a_*$ in $(0,1)$ of the equation $\Psi^*_1(t)=0.$ As a  consequence, we see that
\begin{align*}
	\mathcal{A}_{f_a}((\sqrt{17}-3)/4)=1+16(-3+\sqrt{17})^4(\lambda_1-\lambda)\frac{(1-a^2)^4}{\left(8+(-13+3\sqrt{17})a^2\right)^4}
\end{align*}
which is bigger than $1$ in case $\lambda_*>\lambda.$ This proves the sharpness assertion and the proof of theorem is completed.
\end{proof}
\begin{proof}[\bf Proof of Theorem \ref{th-3.3}]
	Let $|a_0|=a\in(0,1)$. In view of \eqref{e-11.18}, \eqref{e-5.1}
 and Lemma \ref{lem-44.1}	with $N=1, $ we obtain 
 \begin{align*}
 	\mathcal{J}_{f}(r)\leq& \frac{r+a}{1+ra}+\frac{r}{1-r^2}\left(1-\left(\frac{r+a}{1+ra}\right)^2\right)+\frac{(1-a^2)r^2}{1-r}+\frac{221-43\sqrt{17}}{64}\times\\&\quad\frac{(1-a^2)^2r^2}{(1-r^2)(1-r^2a^4)}+\mu\left(\frac{(1-a^2)^2r^2}{(1-r^2)(1-r^2a^4)}\right)^2\\&=\frac{r+a}{1+ra}+\frac{r(1-a^2)}{(1+ar)^2}+\frac{(1-a^2)r^2}{1-r}+\frac{221-43\sqrt{17}}{64}\frac{(1-a^2)^2r^2}{(1-r^2)(1-r^2a^4)}\\&\quad+\mu\left(\frac{(1-a^2)^2r^2}{(1-r^2)(1-r^2a^4)}\right)^2:=\mathcal{J}^*_{f}(r).
 \end{align*}
Since $\mathcal{J}^*_{f}(r)$ is an increasing function of $r$ , we have for $\mathcal{J}^*_{f}((\sqrt{17}-3)/4)\leq 1$  that
\begin{align*}
\mathcal{J}_{f}(r)&\leq\mathcal{J}^*_{f}((\sqrt{17}-3)/4)\\&=\frac{(-3+\sqrt{17})+4a}{4+\left(-3+\sqrt{17}\right)a}+\frac{(-3+\sqrt{17})^2(1-a^2)}{4(7-\sqrt{17})}+\frac{4(-3+\sqrt{17})(1-a^2)}{\left(4+(-3+\sqrt{17})|a_0|\right)^2}\\&\quad+\frac{(-3+\sqrt{17})^2(221-43\sqrt{17})(1-a^2)^2}{16(-5+3\sqrt{17})\left(8+(-13+3\sqrt{17})a^4\right)}+\frac{16\;\mu\;(-3+\sqrt{17})^4(1-a^2)^4}{(5-3\sqrt{17})^2\left(8+(-13+3\sqrt{17})a^4\right)^2}\\&=1-\dfrac{8(1-a)^3E_3(a)}{(7-\sqrt{17})(4+(-3+\sqrt{17})a)^2(8+(-13+3\sqrt{17})a^2)^4},
\end{align*}
where 
\begin{align}\label{ee-2.8}
	E_3(a):=&Q_1+Q_1a+Q_1a^2+Q_1a^3+Q_1a^4+Q_1a^5+Q_1a^6+Q_1a^7+Q_1a^8\\&\quad+Q_1a^9-128\mu(1-a)\nonumber(1+a)^4Q_{10},
\end{align}
with
\begin{align*}
	\begin{cases}
		&Q_0=(2689648-651024\sqrt{17}),\;Q_1=(321568-77024\sqrt{17}),\\&Q_2=4(-1166421+282995\sqrt{17}),\; Q_3=(762484-184748\sqrt{17}),\\&Q_4=4(-974531+236341\sqrt{17}), \;Q_5=-8(-181177+43951\sqrt{17}),\\&Q_6=(14798109-3589099\sqrt{17}),\;Q_7=(-2416805+586147\sqrt{17}),\\&Q_8=4(-3664887+888865\sqrt{17}),\;Q_9=4(-1952633+473583\sqrt{17}),\\&Q_{10}=(3580-868\sqrt{17})+4(-3187+773\sqrt{17})a-(-11351+2753\sqrt{17})a^2.
	\end{cases}
\end{align*}
Following the idea of the proof of Theorem \ref{th-3.2} we confirm that the function $E_3(a)$ has exactly one stationary point $a_{**}\approx 0.565671$ in $(0,1). $ In fact, through the calculation of Mathematica, we obtain that the number $a_{**}\approx 0.565671$  is the unique root in $(0,1)$ of the equation $	E^{\prime}_3(a)=0$ also.Where 
\begin{align*}
	E^{\prime}_3(a)=E_1(a)-128\times16.0824\dots\times(1+a)^3E_2(a)
\end{align*}

Now we plug the value the value of $\mu$ defined in the statement of the theorem into the expression of $E_3(a_{**}).$ Which gives
\begin{align*}
	E_3(a_{**})=\frac{8 \Psi^*_2(a_{**})}{E_2(a_{**})}
\end{align*}
where $\Psi_2^*(t)$ is given by \eqref{e-3.2}.
It can be easily seen that $a_{**}$ is the unique root in $(0,1) $ of the equation  $\Psi_2^*(t)=0.$ Thus  $E_3(a_{**})=0$ and  $E_3^{\prime}(a_{**})=0$. Besides this observation, we have $E_3(0)=4287.01>0$ and $E_3(1)=10949.3>0.$ Consequently, $E_3(a)\geq0 $ in the interval $(0,1) $, which proves that $\mathcal{J}_{f}(r)\leq 1$ for $r\leq(\sqrt{17}-3)/4.$  
By the similar argument of the previous theorem we can easily prove the sharpness assertion of $\mu
$ with the help of the function $f_a$.
\end{proof}
\noindent{\bf Acknowledgment:} The authors would like to express their great indebtedness to the anonymous referees for their elaborate comments and valuable suggestions, which significantly improved the presentation of the paper. The second author is supported by UGC-JRF (NTA Ref. No.: $ 201610135853 $), New Delhi, India. \vspace{2mm}

\noindent\textbf{Compliance of Ethical Standards}\\

\noindent\textbf{Conflict of interest} The authors declare that there is no conflict  of interest regarding the publication of this paper.\vspace{1.5mm}

\noindent\textbf{Data availability statement}  Data sharing not applicable to this article as no datasets were generated or analysed during the current study.

\end{document}